\documentclass[11pt,a4paper]{amsart}
\usepackage{amssymb,amsmath,amsthm}

\newtheorem{theorem}{Theorem}
\newtheorem*{theorem*}{Theorem}

\newtheorem{lemma}{Lemma}

\renewcommand{\geq}{\geqslant}
\renewcommand{\leq}{\leqslant}

\newcommand{\R}{\mathbb{R}}

\newcommand{\E}{\mathcal{E}}

\newcommand{\tr}{\mathrm{tr }}

\newcommand{\la}{\langle}
\newcommand{\ra}{\rangle}

\begin{document}

\title{A short proof for the characterisation of tight frames}

\author{Gergely Ambrus}
\email[G. Ambrus]{ambrus@renyi.hu}

\address{\'Ecole Polytechnique F\'ed\'erale de Lausanne, EPFL SB  MATHGEOM DCG
Station 8, 
CH-1015 Lausanne, Switzerland\\
and\\
Alfr\'ed R\'enyi Institute of Mathematics, Hungarian Academy of Sciences, Re\'altanoda u. 13-15, 1053 Budapest, Hungary}
%\date{\today}
\thanks{Research was supported by OTKA grants 75016 and
76099.}
%\keywords{Orthogonal crosspolytopes, Gaussian vectors.}
%\subjclass[2010]{52A40(primary), and 60D05(secondary)}

\maketitle

\begin{abstract}
With the aid of utilising tensor products, we give a simplified proof to the fundamental theorem of Benedetto and Fickus \cite{BF03} about the existence and characterisation of finite, normalised tight frames. We also establish unit-norm tensor resolutions for symmetric, positive semi-definite matrices.
\end{abstract}

\section{Introduction}

A set of vectors $u_1, \ldots, u_N$ in $\R^n$ is called a {\em tight frame}, if there exists a constant $\lambda \in \R$ so that for every vector $x \in \R^n$,
\[
\lambda x = \sum_{i=1}^N u_i \la x, u_i \ra
\]
holds; using tensor product notation, the above equation transforms to
\begin{equation}\label{fra}
\sum_{i=1}^N u_i \otimes u_i = \lambda \, I_n.
\end{equation}
In the special case when all the vectors $u_i$ are of norm 1, we are talking about a {\em normalised} (or {\em unit-norm}) {\em tight frame}. By comparing traces in \eqref{fra}, it immediately follows that in this case, $\lambda = N/n$.

The theory of frames was initiated by Duffin and Schaffer \cite{DS52} in 1952. For the rich history of the subject, we refer the interested reader to the articles \cite{BF03, CFKLT06}.  
The most well-known example is the orthonormal system in $\R^n$; here $n = N$ holds. It is an essential question whether normalised tight frames exist for all $n$ and $N$ satisfying $N \geq n$. This has been answered affirmatively by \cite{GKK01} and \cite{Z01} in 2001, by providing an explicit construction. A bit later, a new proof was found by Benedetto and Fickus \cite{BF03}. The authors do not provide an explicit construction, rather they prove that normalised tight frames are minimisers of an adequately chosen potential function. We are mainly interested in this result.

To any vector system $(u_i)_1^N$ of $N$ vectors in $\subset \R^n$, we associate the {\em frame potential} by
\begin{equation}\label{framepot}
FP\left((u_i)_1^N\right) = \sum_{i,j = 1} ^N \la u_i, u_j \ra ^2.
\end{equation}

\noindent
The main result of \cite{BF03} asserts the following (see Theorem 7.1 therein).

\begin{theorem} \label{thmframe}
Among the systems of $N$ vectors on the unit sphere $S^{n-1}$, every local minimiser of the frame potential is also a global minimiser. Furthermore, these extremal vector systems are orthonormal sets if $N \leq n$, and normalised tight frames if $N \geq n$.
\end{theorem}

The goal of this note is to give an elegant and short proof to the above result based on the method of \cite{BF03}. Our presentation is more transparent though, thanks to the notion of tensor products. For the sake of simplicity, we chose to work in the real setting; the proofs translate to the complex case without difficulty.

\section{Tensor products and positive definite matrices}

Let $u, v$ be  $n$-dimensional real vectors with coordinates $u_1, \ldots, u_n$ and $v_1, \ldots, v_n$, respectively. The {\em tensor product} of $u$ and $v$ is the $\R^n \rightarrow \R^n$ linear map $u \otimes v$  satisfying
\[
(u \otimes v) z =  u \la z,v \ra
\]
for every $z \in \R^n$. In matrix form, $u \otimes v$ is the $n \times n $ matrix with entries
\[
(u \otimes v)_{ij} = u_i v_j\,.
\]
As an immediate consequence of the definition, we derive that
\[
\tr (u \otimes v) =  \la u, v \ra\,.
\]
If $u$ is a unit vector, then $u \otimes u$ is the projection onto to linear span of $u$. 

The natural inner product on the space of $n \times n$ real matrices is given by
\[
\la A, B \ra := \tr (A B^\top) = \sum_{i,j =1 }^n a_{ij} b_{ij},
\]
where $A, B \in \R^{n \times n}$. It induces the {\em Hilbert-Schmidt norm}:
\[
\|A\|_{HS}  = \|A\|= \la A, A \ra ^{1/2} = \left(\sum_{i,j=1}^n a_{ij}^2 \right)^{1/2}\,.
\]
Inner products and tensor products interact nicely:
\begin{equation}\label{tuv}
\la T, u \otimes v \ra = \sum_{i,j} t_{ij} u_i v_j =  \la Tv, u \ra = \la T^\top u, v \ra\,.
\end{equation}
In particular, for vectors $u_1,u_2,v_1,v_2 \in \R^n$,
\begin{equation}\label{tensprod}
\la u_1 \otimes v_1, u_2 \otimes v_2 \ra = \la u_1, u_2 \ra \la v_1, v_2 \ra.
\end{equation}

Let now $M$ be a positive semi-definite, symmetric, $n \times n$ real matrix. By the spectral theorem, there exists an orthonormal basis $u_1, \ldots, u_n$ of $\R^n$ consisting of eigenvectors of $M$. If $\lambda_i$ denotes the (non-negative) eigenvalue corresponding to the eigenvector $u_i$ for $i =1, \ldots, n$, then
\[
M = \sum_{i=1}^n \lambda_i u_i \otimes u_i \, = \sum_{i=1}^n (\sqrt{\lambda_i} u_i) \otimes (\sqrt{\lambda_i} u_i);
\]
this is the spectral resolution of $M$.

In this section, we prove that such a decomposition of the matrix $M$ also exists if we require the vectors in the summands to have the same norm, which, naturally, depends on the trace of $M$. Although we will not make a direct use of this result in the course of  the proof of Theorem~\ref{thmframe}, it provides an important insight to the problem.

A {\em centred ellipsoid} $\E$ in $\R^n$ is defined by
\begin{equation}\label{ell}
\E = \{ x \in \R^n: x^\top M x = 1\}
\end{equation}
for a symmetric, positive semi-definite $n \times n$ real matrix $M$ (note that we use the term for hollow shell). If the matrix $M$ is non-singular, then $\E$ is non-degenerate; in this case, $\E = M^{-1/2} S^{n-1}$, where $M^{1/2}$ is the positive definite square-root of $M$.

\begin{lemma}\label{elllemma}
Every non-degenerate, centred ellipsoid in $\R^n$ contains an orthogonal system of vectors of the same length.
\end{lemma}

\begin{proof}
We are going to proceed by induction on $n$. The statement is clearly true for $n=1$. Let the ellipsoid $\E$ be defined as in \eqref{ell}.  Notice that if $z_1, \ldots, z_n \subset \E$ are pairwise orthogonal vectors of norm $\rho$, then
\[
\rho^2 I_n  = \sum_{i=1}^n z_i \otimes z_i
\]
and
\begin{equation}\label{nuell}
\rho^2 M = \sum_{i=1}^n (M^{1/2} \,z_i) \otimes (M^{1/2} \, z_i),
\end{equation}
where the vectors $M^{1/2} \, z_i $ are of norm 1. Thus, by comparing traces in \eqref{nuell}, we obtain that
$\rho = \sqrt{n/  \tr \,M}$; in particular, the constant $\rho$ uniquely belongs to the ellipsoid $\E$.

Assume now that the statement holds for every $(n-1)$--dimensional centred ellipsoid. For every $x \in \E$, $x^\perp \cap \E$ is a centred $(n-1)$--dimensional ellipsoid, thus, it contains a scaled copy of an orthonormal basis with scaling factor $\rho(x)$, where  $\rho(x)$ is a continuous function of $x$. For $x$ being the  minimiser or the maximiser of the Euclidean norm on $\E$, $\rho(x) \geq |x|$ and $\rho(x) \leq |x|$ holds, respectively. Thus, by the intermediate value theorem, there exists a point $x_0$ where $|x_0| =  \rho (x_0)$, yielding an appropriate vector system.
\end{proof}

\begin{theorem}\label{lemmatrix}
Let $M$ be a positive semi-definite, symmetric, real $n \times n$ matrix with trace $N$, where $N \geq n$. Then there exist $N$ unit vectors $(v_i)_1^N$ in $\R^n$, so that
\begin{equation}\label{mvi}
M = \sum_{i=1}^N v_i \otimes v_i.
\end{equation}
\end{theorem}

\begin{proof}

Without loss of generality, by restricting the vectors $(v_i)$ to lie in the orthogonal complement of the nullspace of $M$, we may assume that $M$ is positive definite.

Next, we show by induction on $N$ that it suffices to consider the case $n = N$. Indeed, suppose that $N>n$, and let $\lambda_n$ be the largest eigenvalue of $M$ with a corresponding eigenvector $w$ of norm 1. By the trace condition, $\lambda_n>1$, hence, we can reduce the problem to the positive definite matrix
\[
M - w \otimes w,
\]
which has trace $N -1$.

Thus, we suppose that $n=N$. By Lemma~\ref{elllemma}, there exist vectors $(z_i)_1^n$ on the ellipsoid $M^{-1/2} S^{n-1}$ and a constant $\rho$  which satisfy \eqref{nuell}:
\[
\rho^2 M = \sum_{i=1}^n (M^{1/2} z_i) \otimes (M^{1/2} z_i).
\]
Since $M^{1/2} z_i$ is a unit vector for every $i$, and $\tr\, M = n$, it follows that $\rho =1$, thus, the above resolution is of the form that we seek.
\end{proof}

Note that a given matrix $M$ may have many representations of the form~\eqref{mvi}. The essential question arises: are the natural topology on the space of $N$-tuples of unit vectors, and the natural topology on the space of $n \times n$ positive semi-definite matrices with trace $N$, induced by each other under this map? The answer is negative. The easiest counterexample is the $n=N=2$ case. Here, the identity matrix has infinitely many resolutions: $I_2 = u_1 \otimes u_1 + u_2 \otimes u_2$, whenever $(u_1, u_2)$ is an orthonormal system. However, any other positive semi-definite $2 \times 2$ real matrix has a unique representation, up to sign changes and the permutation of coordinates. Thus, given a vector system $(u_1, u_2)$ which represents the identity, we may find a small modification $I'$ of $I_2$ so that no small change of the vectors $u_i$ yields a resolution of $I'$. In higher dimensions, the induced topology on the matrix space gets more rich; this sheds light on the non-triviality of Theorem~\ref{thmframe}.

\section{Proof of the characterisation result}

\begin{proof}[Proof of Theorem~\ref{thmframe}]
By intersecting $S^{n-1}$ with the linear span of the vectors, we may assume that $N \geq n$.

Let $G$ denote the Gram matrix corresponding to the vector system $(u_i)_1^N$, that is, the $N \times N$ matrix whose $(i,j)$th entry is $\la u_i , u_j \ra$. If $L$ denotes the $N \times n$ matrix with rows $u_1,\ldots, u_N $, then
\[
G = L L^\top.
\]
On the other hand, the {\em frame operator} $S$ is defined by
\begin{equation}\label{seq}
S = L^\top L = \sum_1^N u_i \otimes u_i.
\end{equation}
The frame potential of the vector system $(u_i)$ is $\|G\|^2 = \tr \, G^2$. Thus, using the commutativity of the trace of a product of matrices,
\begin{equation}\label{GS}
FP\left((u_i)_1^N\right) = \tr \, G^2 = \tr (L L^\top L L^\top) = \tr (L^\top L L^\top L) = \tr \, S^2 = \| S \| ^2.
\end{equation}
Equation \eqref{seq} implies that $\tr \,S = N$. Therefor, by the Cauchy-Schwarz inequality,
\[
\| S \|^2 \geq \sum_{i=1}^n S_{ii}^2 \geq \frac {N^2}{n}\,.
\]
Furthermore, equality can hold only if $S$ is a constant multiple of $I_n$, that is,
\[
\frac N n I_n = \sum_{i=1}^N u_i \otimes u_i.
\]
Comparing with \eqref{fra}, we conclude that the global minimiser vector systems of the frame potential are exactly the normalised tight frames.

If $n= N$, and $(u_i)_1^n$ is a global minimiser of the frame potential, then \eqref{GS} shows that $\|G\|^2 = n$. Taking into account the fact that the diagonal entries of $G$ are $1$, we conclude that $(u_i)_1^n$ must necessarily be an orthonormal sequence.

The rest of the work goes into proving that every local minimiser of the frame potential is also a global minimiser. In light of the remark at the end of the previous section, this is not a trivial statement.

Assume that the vector system $u_1, \ldots, u_N$ in $S^{n-1}$ is a local minimiser of the frame potential.
We shall introduce a {\em local change} as follows: for each $i \in [N]$, let $v_i \in S^{n-1}$ be a vector orthogonal to $u_i$, let $\delta_i \in \R$ a constant close to 0, and define
\[
\widetilde{u_i} = u_i (\delta_i) = \cos \delta_i u_i + \sin \delta_i v_i = u_i + \delta_i v_i - \frac {\delta_i^2} 2 u_i + O(\delta_i^3).
\]
Then, $u_i \otimes u_i$ changes to
\begin{equation}\label{uitensor}
\widetilde{u_i} \otimes \widetilde{u_i} =  u_i \otimes u_i + \delta_i(u_i \otimes v_i + v_i \otimes u_i) + \delta_i^2 (v_i \otimes v_i - u_i \otimes u_i) + O(\delta_i^3),
\end{equation}
thus, the first order change is $\delta_i(u_i \otimes v_i + v_i \otimes u_i)$. Our goal is to find a local modification that decreases $\| S \|^2 = \la S, S \ra$. Let $\widetilde S = \sum \widetilde{u_i} \otimes \widetilde{u_i}$.
For a fixed $i \in [N]$, changing only $u_i$ while keeping the other vectors fixed yields

\[
\|\widetilde S \|^2 = \|S\|^2 + 2 \delta_i \la S, u_i \otimes v_i + v_i \otimes u_i \ra + O(\delta_i^2).
\]
The derivative at $\delta_i=0$ with respect to $\delta_i$ is, by \eqref{tuv},
\[
2 \la S, u_i \otimes v_i + v_i \otimes u_i \ra = 4 \la S u_i, v_i \ra.
\]
This is zero for any unit vector $v_i \in u_i^\top$ if and only if $S u_i$ is a constant multiple of $u_i$, that is, $u_i$ is an eigenvector of $S$. Hence, a vector system may only be a local minimiser, if every vector is an eigenvector of the frame operator.

Therefore, for any local minimiser vector system, there exist $\lambda_1, \ldots, \lambda_N \in \R$  so that $ S \,u_i = \lambda_i u_i$ holds for every $i \in [N]$ . This induces a partition of the system $(u_i)$: for an eigenvector $\lambda$ of $S$, let $U(\lambda)$ denote the set of vectors $u_i$ for which $S \, u_i = \lambda u_i$. Since $S$ is positive semi-definite and symmetric, $U(\lambda) \perp U(\mu)$ whenever $\lambda$ and $\mu$ are distinct. Assume that $\lambda$ is the largest of the eigenvalues for which $U(\lambda)$ is non-empty. By re-numbering, we may assume that the vectors in $U(\lambda)$ are exactly $u_1, \ldots u_k$. Let $H = \textrm{lin}\, U(\lambda)$, the linear span of the vectors $(u_i)_1^k$. Then, the restriction of $S$ to $H$ is $\lambda \, I_H$, a multiple of the identity. On the other hand, \eqref{seq} shows that
\[
\lambda \, I_H = \sum_{i=1}^k u_i \otimes u_i,
\]
thus, $\lambda \dim H = k$. Since $\lambda$ is the largest eigenvalue, it is greater than 1, which implies that $\dim H < k$. Therefore, the vectors $(u_i)_1^k$ are linearly dependent: there exist real constants, not all of which are 0, so that
\begin{equation}\label{uizero}
0 = \sum_{i=1}^k c_i u_i.
\end{equation}
If $(u_i)$ is not a global minimiser, there exists an index $j$, so that $S u_j = \mu u_j$, where $\mu < \lambda$. Set $v = u_j$; then $\la v, u_i \ra = 0$ for every $i = 1, \ldots, k$. For $i = 1, \ldots, k$, let $v_i = v$ and $\delta_i = c_i \delta$, where $\delta$ is a small positive number; for $i = k+1, \ldots, N$, set $\delta_i = 0$. Making use of \eqref{tuv}, \eqref{tensprod}, \eqref{uitensor} and \eqref{uizero} leads to
\begin{align*}
\| \widetilde S \|^2 & =  \| S \|^2 + 2 \delta \sum_{i=1}^k c_i \la S,u_i \otimes v + v \otimes u_i \ra + 2 \delta^2 \sum_{i=1}^k c_i^2 \la S, v \otimes v - u_i \otimes u_i \ra \\ &\quad + \delta^2 \sum_{i,j=1}^k c_i c_j \la  u_i \otimes v + v \otimes u_i, u_j \otimes v + v \otimes u_j \ra + O(\delta^3)\\
&=  \| S \|^2 - 2 \delta^2 ( \lambda - \mu ) \sum_{i=1}^k c_i^2  + 2 \delta^2 \sum_{i,j=1}^k c_i c_j \la u_i, u_j \ra \la v,v \ra + O(\delta^3)\\
&= \| S \|^2 - 2 \delta^2 ( \lambda - \mu) \sum_{i=1}^k c_i^2   + 2 \delta^2 \left| \sum_{i=1}^k c_i u_i \right|^2 + O(\delta^3) \\
&= \| S \|^2 - 2 \delta^2 (\lambda -  \mu) \sum_{i=1}^k c_i^2  + O(\delta^3).
\end{align*}
Recalling that $\mu < \lambda$, this shows that for sufficiently small $\delta$, $\|\widetilde{S}\| < \| S \|$; therefore, $(u_i)_1^n$ cannot be a local minimiser of the frame potential.
\end{proof}

\end{document}